\documentclass[12pt,twoside]{article}
\usepackage{latexsym, amssymb,amsthm,amsmath,xypic,lscape,epsfig,setspace,tikz,slashed,color}
\usepackage[retainorgcmds]{IEEEtrantools}
\usetikzlibrary{decorations.markings}
 \usetikzlibrary{arrows}
\usepackage[pdftex]{hyperref}
\usepackage[margin=10pt,font=small,labelfont=bf]{caption}

\theoremstyle{plain}
\newtheorem{theorem}{Theorem}[section]
\newtheorem{proposition}[theorem]{Proposition}
\newtheorem{lemma}[theorem]{Lemma}
\newtheorem{question}[theorem]{Question}
\newtheorem{corollary}[theorem]{Corollary}

\theoremstyle{definition}
\newtheorem{definition}[theorem]{Definition}
\AtBeginEnvironment{definition}{%
  \pushQED{\qed}%
}
\AtEndEnvironment{definition}{\popQED\endexample}

\newtheorem{remark}[theorem]{Remark}
\AtBeginEnvironment{remark}{%
  \pushQED{\qed}%
}
\AtEndEnvironment{remark}{\popQED\endexample}

\newtheorem{example}[theorem]{Example}
\AtBeginEnvironment{example}{%
  \pushQED{\qed}%
}
\AtEndEnvironment{example}{\popQED\endexample}

\newenvironment{renumerate}%
{%
\begin{enumerate}}%
{\end{enumerate}%
}%

{\vskip6pt%
\noindent%
{\it Remarks}. %
\begin{renumerate}}%
{\end{renumerate}\vskip6pt}

\makeatletter
\def\Ddots{\mathinner{\mkern1mu\raise\p@
\vbox{\kern7\p@\hbox{.}}\mkern2mu
\raise4\p@\hbox{.}\mkern2mu\raise7\p@\hbox{.}\mkern1mu}}
\makeatother

\newcommand{\R}{\text{${\mathbb R}$}}
\newcommand{\C}{\text{$\mathbb C$}}

\newcommand{\Z}{\text{$\mathbb Z$}}

\newcommand{\Q}{\text{$\mathbb Q$}}

\renewcommand{\frak}[1]{\text{$\mathfrak{#1}$}}
\renewcommand{\tilde}{\widetilde}

\newcommand{\M}{\text{$\mathcal{M}$}}

\renewcommand{\bar}{\overline}

\newcommand{\del}{\text{$\partial$}}

\newcommand{\tensor}{\otimes}

\newcommand{\mc}[1]{\text{$\mathcal{#1}$}}

\newcommand{\noqed}{\let\qed\relax}

\newcommand{\IP}[1]{\langle #1 \rangle}

\definecolor{darkred}{rgb}{0.7,0,0}
\definecolor{darkgreen}{rgb}{0,0.5,0}
\definecolor{darkblue}{rgb}{0,0,0.5}
\definecolor{darkorange}{rgb}{0.3,0.6,0.2}
\definecolor{darkyellow}{rgb}{0.75,0.75,0.2}

\date{} \usepackage{color} \definecolor{tocolor}{rgb}{.1,.1,.5}
\definecolor{urlcolor}{rgb}{.2,.2,.6}
\definecolor{linkcolor}{rgb}{.1,.1,.6}
\definecolor{citecolor}{rgb}{.6,.2,.1}
\hypersetup{backref=true, colorlinks=true, urlcolor=urlcolor,
  linkcolor=linkcolor, citecolor=citecolor}

\begingroup
\hypersetup{linkcolor=blue}

\endgroup

\numberwithin{equation}{section}

\begin{document}

\title{Massey products, sphere bundles and T-duality}
\author{Gil R. Cavalcanti\thanks{{\tt gil.cavalcanti@gmail.com}, Department of Mathematics,
Utrecht University, The Netherlands}}
\maketitle

\abstract{We study spherical T-duality for iterated sphere bundles. We show that for a class of iterated sphere bundles the cohomological data contained in its Gysin sequences can be repackaged into data for a vanishing Massey product. We further show that if these bundles are endowed with an integral cohomology class of transgressive degree one, then they have a T-dual iterated sphere bundle, namely, the one  associated to the same Massey product read backwards.}
\vskip12pt
\noindent
MSC classification 2020: 	55R25, 55N45, 55R20, 57R20, 57R19, 57R22, 57N65, 81T30.\\
Subject classification: Differential Topology.\\
Keywords: Massey products, sphere bundles, T-duality.\\

\tableofcontents

\section{Introduction}

We uncover a surprising  relation between three  concepts that a first sight bear no connection: Massey products, iterated sphere bundles and T-duality. Namely we show that:
\begin{enumerate}
\item  iterated sphere bundles are intimately related to the data needed to compute Massey products at the cochain level  and
\item iterated sphere bundles have T-duals which are described by reading the associated Massey product backwards.
\end{enumerate}
Let us say a few words about each of these concepts separately.

Massey products \cite{Mas58} belong to the realm of Algebraic Topology. They are defined for differential graded algebras (DGAs) and their cohomology. The binary Massey product is just the induced product in cohomology which is a first order compatibility condition between the algebra product and the differential. Higher Massey products are obtained from higher operations on cochains which are built on the fact that the zero  cohomology class can be represented by nonzero cochains. As such, higher Massey products are  a clear obstruction to the existence of a quasi-isomorphism between a DGA and its cohomology and  appear, for example, in questions of formality of manifolds \cite{Sul77}.

Iterated sphere bundles belong to the realm of differential geometry. They consist of  a finite collection of  bundles over bundles  over a fixed manifold 
\[E_N \to  E_{N-1} \to \dots \to E_1 \to  E_0 =M,\]
such that for each $i$, $E_i \to E_{i-1}$ is an oriented sphere bundle with odd dimensional fibers. Cohomologically, one can describe $E_N$ via a sequence of Gysin sequences. We will focus on ``level-one'' iterated sphere bundles (see Definition \ref{def:level i}). One consequence of this condition is that the fibers of $E_N \to E_0$ are just the product of the spheres used in the iteration process.  Notice however that this process allows for (higher) monodromy. For example,  level-one iterated circle bundles are more general than principal torus bundles and are particular cases of  principal affine torus bundles as studied in \cite{Ba15}.

Possibly the least familiar mathematical concept we use  is spherical T-duality \cite{BEM15,LSW16,CHU24}. T-duality, or target space duality, is a notion that comes from string theory where it  yields an equivalence between theories compactified on possibly distinct  target spaces \cite{Bus87,RV92}.  The original physical notion of  T-duality for circle bundles found a solid mathematical foundation in \cite{BEM04,CG11}  and has been extended in many directions  including  principal and  affine torus bundles \cite{BHM04,Ba15} (which allow for monodromy at the price of introducing local systems) and to sphere bundles \cite{BEM15,LSW16,CHU24}.  T-duality for principal torus bundles and spherical T-duality  are  often described in a colloquial way as  ``an exchange of topological information from a background cohomology class with topological information present in the characteristic classes of the sphere bundle''.  Mathematically, T-dual spaces are related by three isomorphisms:  one in twisted cohomology, one in twisted K-theory \cite{BEM04,BEM15} and one of Courant algebroids \cite{CG11,CHU24}. Therefore, regardless of the physical applications, finding  T-dual spaces yields precise  mathematical relations between T-dual spaces.

In this paper we will show that there is a near equivalence between Massey product data (for vanishing Massey products) and  level-one iterated sphere bundles with a background closed form. We will also show that given an iterated sphere bundle $(E_i\to E_{i-1})_{i=0,\dots N}$ if we read its associated  Massey product ``backwards'' we can construct a second (possibly topologically distinct) iterated sphere bundle and that this sphere bundle is T-dual to the original one. The original  description of T-duality as an exchange of topological information from a background cohomology class with topological information from characteristic classes of the sphere bundle obtains a clear re-interepretation in this context. The new set-up has many more moving parts and is not as immediately stated, but, once the relation is seen, it cannot be unseen. This interpretation had not appeared in the literature yet, even in the context of affine torus bundles, which are known to be relevant to string theory.

This paper is organised as follows. In Section  \ref{sec:Massey}, we introduce Massey produtcs and Massey product data. We finish that section explaining how to repack Massey product data in  matrix notation, as  done in \cite{May69}. In Section \ref{sec:interated bundles}, we introduce iterated sphere bundles, and the tools to study their cohomology. We then  state and prove  the precise relation between transgressive level-one iterated sphere bundles and vanishing Massey product data. In Section \ref{sec:order reversal} we explain that reading the Massey product data backwards also yields iterated sphere bundles. This is in itself not automatic because we work over integral cochains which do not form  a graded commutative algebra, so flipping the order of the factors in a Massey product is not an allowed operation. In Section \ref{sec:T-duality} we introduce the  mathematical definition of spherical T-duality and show that the order reversal duality of Massey products is T-duality. We give a few concrete examples in Section \ref{sec:examples}  including a differential model the complement of a nodal fiber in an elliptic fibration in four dimensions. In Section \ref{sec:beyond}, we indicate a way to extend the setup beyond interated sphere bundles.

The author thanks Marius Crainic for useful conversations and David Baraglia for valuable feedback on an early version of the manuscript.

\section{Massey products}\label{sec:Massey}

Massey products are standard algebraic operations introduced by Massey back in the 50's \cite{Mas58}. We introduce them here for completeness and to fix the notation we will use in the rest of the paper. We start with the usual description of Massey products and finish with the more compact matrix notation for  the concept \cite{May69}.  The home to Massey products are differential graded algebras, which we introduce next.

\begin{definition}
A \emph{differential graded algebra} (DGA) over a ring $R$ is a graded algebra  $\mc{M}^\bullet$ endowed with an $R$-linear degree 1 map, $d \colon \mc{M}^k\to \mc{M}^{k+1}$ (for all $k$), such that
\begin{itemize}
\item ($d$ is a differential) $d^2 =0$ ,
\item ($d$ is a derivation) $d(a b) = (d a) b+ (-1)^{|a|}a (d b)$.
\end{itemize}
\end{definition}

Since $d^2=0$ we can define the cohomology of $(\mc{M}^\bullet,d)$. Because $d$ is $R$-linear and a derivation, $H^\bullet(\mc{M},d)$ is in itself an $R$-algebra.

Massey products are $n$-ary operations defined inductively in the cohomology algebra of a DGA. A convenient way to keep track of the many signs that appear in the definition is by introducing the sign change operator
\[\bar{a}: = (-1)^{|a|+1}a.\] 
It is good to spell out how this operator behaves when  composed with other common operators:
\begin{equation}\label{eq:properties}
\overline{ab} = - \bar{a}\bar{b},\quad d\bar{a} = - \overline{da},\quad d(ab) = (da)b - \bar{a} db.
\end{equation}
With the sign change operator at hand we are ready to define Massey products.
\begin{definition}
The \emph{binary Massey product} is the degree-zero operation
\[\IP{\cdot,\cdot}\colon H^\bullet(\mc{M})\times H^\bullet(\mc{M}) \to H^\bullet(\mc{M}), \quad \IP{[a_1^0],[a_2^1]} := [\bar{a_1^0} a_2^1].\]
\end{definition}
\begin{definition}
The \emph{triple Massey product} is a partially defined degree $-1$ operation,
\[\IP{\cdot,\cdot,\cdot}\colon H^\bullet(\mc{M})\times H^\bullet(\mc{M}) \times H^\bullet(\mc{M}) \to H^\bullet(\mc{M}).\]
The triple product  $ \IP{[a_1^0],[a_2^1],[a_3^2]}$ is only defined if the binary products $ \IP{[a_1^0],[a_2^1]}$  and $ \IP{[a_2^1],[a_3^2]}$ vanish, in which case there are cochains $a_2^0$ and $a_3^1$ such that
\begin{equation}\tag{*}
da_2^0 = \bar{a_1^0} a_2^1,\quad  da^1_3 = \bar{a_2^1}a_3^2.
\end{equation}
With those at hand, the  \emph{triple  Massey product}  is
\[ \IP{[a_1^0],[a_2^1],[a_3^2]} := [\bar{a_1^0} a_3^1 + \bar{a_2^0}a_3^2] + [\bar{a_1^0}] H^\bullet(\M) + H^\bullet(\M)[a_3^2].\]
\end{definition}

Notice that
\begin{itemize}
\item  by the convenient choice of signs, the cochain $\bar{a_1^0} a_3^1 + \bar{a_2^0}a_3^2$ is a cocycle for combinatorial reasons,
\item  changing $a_2^0$  and $a_3^1$ by cocycles  preserves (*), so the choices  of $a_2^0$ and $a_3^1$ are by no means unique and that is why the triple product is not a cohomology class, but a coset.
\end{itemize}
Said another way, the triple product is the set of all cohomology classes of the form  $[\bar{a_1^0} a_3^1 + \bar{a_2^0}a_3^2] $  obtained by making different choices of primitives for the binary products. 

\begin{definition}
The triple Massey product $\IP{[a_1^0],[a_2^1],[a_3^2]} $  \emph{vanishes} if $ 0 \in \IP{[a_1^0],[a_2^1],[a_3^2]} $.
\end{definition}

The triple Massey product brought up two obvious new features absent in the binary product: the fact that is it partially defined and that its values have ambiguity.

\begin{example}\label{ex:many roads to zero}
There is a third feature of triple products which is more easily overlooked. Namely, that there are many ways to achieve a fixed value for the triple Massey product. Indeed, changing the choice of $a_2^0$ and $a_3^1$ by cocycles $b_2^0$ and $b_3^1$ will not change the cohomology class of the final cocycle as long as
\[[\overline{b_2^0} a_3^2 + \overline{a_1^0} b_3^1] = 0.\] 

So, for example, even the triple product $\IP{0,0,0} = 0$ can be expressed in less trivial ways by picking arbitrary cocycles $b_2^0$ and $b_3^1$.
\end{example}

From the quadruple  product onwards there is also the need to make  {\it uniform of choices}.  This last ingredient sets the way for the general Massey products
%
%
\begin{definition}
The \emph{$n$-Massey product} is a partially defined degree $2-n$ operation,
\[\IP{\cdot,\dots,\cdot}\colon (H^\bullet(\mc{M}))^n \to H^\bullet(\mc{M}).\]
The $n$-product  $ \IP{[a_1^0],[a_2^1],\dots,[a_n^{n-1}]}$ is only defined if the $n-1$-products $ \IP{[a_1^0],\dots,[a_{n-1}^{n-2}]}$  and $ \IP{[a_2^1],\dots, [a_n^{n-1}]}$ vanish uniformly, that is, there are elements $a^i_j$ such that
\begin{equation}\label{eq:defining  data}
d a_j^i = \sum_{k=i+1}^{j-1}\bar{a_k^i}a_j^k\qquad \mbox{  for  } 0\leq i<j\leq n,~(i,j)\neq (0,n).
\end{equation}
Then  the \emph{$n$-Massey product} is the collection of  cohomology classes of the form 
\[\IP{[a_1^0],\dots,[a_n^{n-1}]} := \left[\sum_{k=0}^n\bar{a_k^0} a_n^k\right],\]
obtained by different choices of $\{a^i_j\}_{0\lneq i <j\leq n}$ satisfying \eqref{eq:defining data}
\end{definition}

\begin{definition}
\emph{Massey product data} is any collection of cochains $\{a_j^i\} \in \mc{M}$ with $0\leq i<j \leq n$  and $(i,j)\neq (0,n)$ satifying \eqref{eq:defining  data} and \emph{vanishing Massey product data} is any such collection for which \eqref{eq:defining  data}  also holds for $(i,j)= (0,n)$.
\end{definition}

\begin{remark}\label{rem:matrix massey}
As was observed in \cite{May69}, in the discussion so far, there is a matrix multiplication hidden in plain sight: using the cochains $\{a_j^i\}_{0\lneq i <j\leq n}$ we can form a strictly upper triangular matrix $A$:
\[A^i_j = \begin{cases} a^i_j &\mbox{ for } i <j\\
0& \mbox{ for }i\geq j.
\end{cases}
\]
Then the condition that  $\{a_j^i\}_{0\lneq i <j\leq n}$  is Massey product data is equivalent to
\[dA - \bar{A} A = 
\begin{pmatrix}
0 & 0 & \cdots & da^0_n - \sum^{k=i}_i\bar{a_k^0}a_n^k\\
0 & 0 & \cdots & 0\\
\vdots & \vdots& \ddots & \vdots\\
0 & 0 & \cdots & 0
\end{pmatrix},
\]
that is, the only non-zero entry in $dA - \bar{A} A $ is the one in position $(0,n)$  and it represents the $n$-Massey product $-\IP{a^0_1, \dots, a^{n-1}_n}$.
In particular, {\it vanishing Massey product data} is equivalent to a matrix as above with 
\begin{equation}\label{eq:vanishing Massey}
dA - \bar{A} A =  0
\end{equation}
\end{remark}


\section{From sphere bundles to Massey products and back}\label{sec:interated bundles}

In this section we look into iterated oriented odd sphere bundles taking a cohomological perspective. An iterated oriented odd sphere bundle over a manifold $M$ is a finite sequence of manifolds $(E_j)_{j\leq N}$ such that
\begin{itemize}
\item $E_0 = M$,
\item for each $j$, $E_j \to E_{j-1}$ is an oriented sphere bundle with odd dimensional fiber.
\end{itemize}
From now on we will refer to these bundles simply as  {\it iterated sphere bundles}.

Given the inductive nature of the definition of iterated sphere bundles,  one would expect that we can compute the cohomology of each $E_j$ using an inductive process, but that is not so simple. Indeed, since $\pi_j\colon E_j \to E_{j-1}$ is an oriented sphere bundle, say with fibers $S^{2n_j -1}$, the integral cohomology of $E_j$ can be computed from the integral cohomology of $E_{j-1}$ via the Gysin sequence.
\[ \cdots \to H^i(E_{j-1}) \stackrel{\pi_j^*}{\longrightarrow} H^i(E_j) \stackrel{\pi_{j*}}{\longrightarrow}  H^{i-2n_j+1}(E_{j-1})  \stackrel{\cup [\varepsilon^{2n_j}]}{\longrightarrow}  H^{i+1}(E_{j-1}) \to \cdots,\]
where $[\varepsilon^{2n_j}]$ is the integral Euler class of  $E_j \to E_{j-1}$.

Effectively, if $(C^\bullet(E_{j-1};\Z),d )$ denotes the complex of integral singular cochains on $E_{j-1}$, we can extend  $(C^\bullet(E_{j-1};\Z),d )$ by a formal element $\psi_j$ of degree $2n_j-1$ and extend the coboundary operator by declaring $d \psi_j = \varepsilon_{2n_j}$:
\begin{equation}
(C_{\psi_j}(E_{j-1},\Z),d ): = (\wedge\IP{\psi_j} \tensor _\Z C^\bullet(E_{j-1};\Z),d ) \subset (C^\bullet(E_{j};\Z),d ).
\end{equation}
Because  $(C^\bullet(E_{j-1};\Z),d )$ is a graded algebra, this process generates a quasi-isomorphic subcomplex, but because the product on cochains is not graded commutative, $(C_{\psi_j}(E_{j-1};\Z),d )$ is not naturally a graded algebra, which spoils the inductive process if we insist on integer coefficients.

To iterate this process, and relate the space of cochains on $E_j$ with the space of cochains on $M$ we need that the new elements, $\{\psi_j\}$  graded commute with each other and with elements in $C^\bullet(M)$. This is only the case if we take coefficients in a field without torsion, say, the rationals.  Repeating the process outlined above, we have that the  rational cohomology of $E_j$ is isomorphic to the rational cohomology of the complex
\begin{equation}\label{eq:left transgressive complex}
(C_{\Psi_{\leq j}},d ): = (\wedge\IP{\psi_1,\dots, \psi_j} \tensor C^\bullet(M;\Q),d ), \quad j \leq N,
\end{equation}
where $d  \psi_j  = \varepsilon_{2n_j} \in C_{\Psi_{\leq j-1}} \subset C_{\Psi_{\leq j}}$ is a cocycle in that subcomplex.

\begin{remark}[Hirsch extensions]
The process described above in a geometric context  can be carried out purely algebraically and is called a {\it Hirsch extension}. The general  algebraic situation is as follows.  Given a DGA, $(\mc{M},d)$, over $\Q$, a  graded $\Q$ vector space, $V$, and a degree-one map of $\Q$ vector spaces $d^V\colon V \to \mc{M}$ such that the image of $d^V$ lies in the kernel of $d$, one can form $\wedge V \tensor_\Q\mc{M} $ which is again an $\Q$-algebra and use the operators $d$ and $d^V$ together with the Leibniz rule to produce a differential on  $\wedge V \tensor_\Q\mc{M} $ turning it into a DGA.

Even if $(\mc{M},d)$ has some connection to geometry (e.g., is the cochain complex of a space),  DGAs obtained as Hirsch extensions  of $\mc{M}$ may lack a geometric counterpart. 
\end{remark}

\begin{definition}
A \emph{transgressive generating set} for an iterated sphere bundle, $(E_j)_{j\leq N}$, is a choice of extensions by  $\IP{\psi_j}$ with $d  \psi_j= \varepsilon_{2n_j} \in C_{\Psi_{\leq j-1}}$  for $j =1 \dots, N$ as above.
\end{definition}

\begin{remark}
Given a transgressive generating set, the choice of including the transgressive generators on the left in \eqref{eq:left transgressive complex} was arbitrary and one could equally well form the right complex
\begin{equation}\label{eq:right transgressive complex}
(C^R_{\Psi_{\leq j}},d ): = (C^\bullet(M;\Q)\tensor \wedge\IP{\psi_1,\dots, \psi_j}) ,d ), \quad j \leq N.
\end{equation}
\end{remark}

\begin{definition}
The complexes \eqref{eq:left transgressive complex} and \eqref{eq:right transgressive complex} are respectively the \emph{left} and \emph{right transgressive complexes} of the iterated sphere bundle $(E_j)_{j\leq N}$ endowed with the transgressive generating set $\Psi$.
\end{definition}
Because we are working over the rationals, the difference between left and right transgressive complexes is only a collection of signs that arise by commuting transgressive elements over cochains on $M$.

This section we will only use left transgressive generating sets, so for now we will not stress the positioning of the transgressive generators.

\begin{definition}
Given a choice of transgressive generating set, $\Psi = \IP{\psi_1,\dots, \psi_N}$, an element of $C_{\Psi_{\leq j}}$ has \emph{transgressive level $i$} if it lies in
\[\wedge^{\leq i} \IP{\psi_1, ,\dots, \psi_j} \tensor C^\bullet(M;\Q)  \subset C_{\Psi_{\leq j}}.\]
\end{definition}

In particular, elements in $C_{\Psi_{\leq j}}$ of transgressive level one are of the form
\begin{equation}\label{eq:right transgressive}
c^0 +  \psi_1 c^1 + \dots + \psi_j c^j , \qquad \mbox{ with } c^i \in C^\bullet(M;\Z).
\end{equation}

\begin{definition}\label{def:level i}
An  {\it  $i^{th}$ transgressive level} iterated sphere bundle $(E_j)_{j\leq n}$  is one for which the Euler class of each  $E_j \to E_{j-1}$ is represented by an element of transgressive level $i$ for a choice of transgressive generating set.
\end{definition}

We will focus on iterated sphere bundles of {\it transgressive level one}. One feature of the level-one subcomplex is that it can be defined also for cochains with {\it integral coefficients}
\[\wedge^{\leq 1} \IP{\psi_1, ,\dots, \psi_j} \tensor C^\bullet(M;\Z)  \subset C_{\Psi_{\leq j}}(M;\Z).\]
as in this case one does not have to commute cochains in $M$ over transgressive generators to define the differential. In this situation, the difference between left and right level-one transgressive complexes is more subtle than mere sign changes.

A motivation  to look into fibration of transgressive level 1 comes from the case when all spheres are circles.
\begin{example}[Level-one iterated circle bundles]
Iterated circle bundles over a point are nilmanifolds. Therefore, the fiber over $p\in M$ of an iterated circle bundle $E_j \to M$ is a nilmanifold. This nilmanifold is a torus if and only if for each $j$ the Euler class of $E_j \to E_{j-1}$ vanishes when restricted to the fibers of $E_{j-1} \to M$. This happens precisely if the Euler class of $E_j \to E_{j-1}$ has transgressive degree 1 for all $j$. Therefore,  level-one iterated circle bundles are precisely those whose fibers $E_N \to M$ are tori.

Notice however that $E_N \to M$ is not a principal torus bundle as those would correspond to level-zero iterated circle bundles. 
\end{example}

In general, being a level-one iterated sphere bundle implies that the fibers of $E_N \to M$  have the same cohomology algebra as
\[S^{2n_1-1}\times \dots \times S^{2n_N-1},\]
but these conditions are not equivalent as we illustrate next.

\begin{example} Let $M^{2n}$ be a connected $2n$-dimensional compact manifold. For $j=1,2$ let $ \varepsilon_j$ be a cocycle in $C^{2n_j}(M;\Z)$ which is the Euler class of a sphere bundle. Let $\sigma \in C^{2n}(M;\Z)$ be a generator of $H^{2n}(M;\Z)$ and let $E_3 \to M$ be the iterated sphere bundle for which
\begin{itemize}
\item  $E_1\to M$ has Euler class $[\varepsilon_1]$,
\item $E_2 \to E_1$ has Euler class  $[\varepsilon_2]$,
\item $E_3 \to E_2$ has Euler class $[m \psi_1 \psi_2 \sigma]$,
where $d  \psi_i = \varepsilon_i$ and $m$ is an appropriate large integer so that the integral closed cochain $m \psi_1 \psi_2 \sigma$ represents the Euler class of a sphere bundle \cite{BV03}.
\end{itemize}
Then $E_3 \to M$ has fibers $S^{2n_1-1}\times S^{2n_2-1}\times S^{2(n+n_1+n_2)-3}$ but is not a  level-one iterated sphere bundle.
\end{example}

The question we pose is:

\begin{question}
How can one describe a level-one iterated sphere bundle over a fixed base manifold.
\end{question}

%

The first significant step to arrive at an answer to the question above relates cocycles of transgressive degree one to Massey products:

\begin{proposition}\label{prop:iterated => Massey}
Let $(E_j)_{j\leq N}$  be a level-one iterated sphere bundle with respect to a transgressive generating set $\Psi = \{\psi_1, \dots, \psi_N\}$,  with Euler forms $d\psi_j = \varepsilon_{2n_j} \in C_{\Psi_{<j}}(M;\Z)$,  and let $H \in C_{\Psi}^l(M;\Z)$ be a level-one $l$-cocycle so we have
\begin{equation}\label{eq:psi and H}
\begin{aligned}
d\psi_j & = \varepsilon_{2n_j} = a_j^0  +\sum_{i=1}^{j-1} \psi_i a_j^i,\\
H&= a_{N+1}^0 + \sum_{i=1}^{N} \psi_i a_{N+1}^i.
\end{aligned}
\end{equation}
Then, the collection $\{a_j^i\}_{0 \leq i <j\leq N+1}$ is vanishing Massey product data for the product
\[\IP{[a^0_1],\dots, [a^N_{N+1}]}\]
\end{proposition}
\begin{proof}
For convenience of notation, we set $\psi_0 =1$, so  formula \eqref{eq:psi and H} for the Euler forms, $\varepsilon_{2n_j}$, and $H$  is more uniform:
 \begin{equation}\label{eq:euler classes and H}
 \varepsilon_{2n^j} = \sum_{i=0}^{j-1} \psi_i a_j^i \quad \mbox{ and } \quad H =  \sum_{i=0}^{N} \psi_i a_j^i .
 \end{equation}

We can express the collection of Euler forms and the extra form, $H$, as a vector, via a matrix equation. Indeed, if we let
\[\begin{pmatrix}\Psi \\ 0\end{pmatrix} = \begin{pmatrix}\psi_0\\\vdots \\ \psi_N\\0\end{pmatrix}\quad \mbox{ and } \quad \begin{pmatrix} \mathbf{0}\\ \bar{H}\end{pmatrix} = \begin{pmatrix}0\\\vdots \\  0 \\ \bar{H}\end{pmatrix} ,\]
then the expression \eqref{eq:euler classes and H} for the Euler forms and $H$ becomes
\begin{equation}\label{eq:the matrix A}
d\begin{pmatrix}\Psi \\0 \end{pmatrix}^t +\begin{pmatrix}\mathbf{0} \\H \end{pmatrix}^t= \begin{pmatrix}\Psi \\0 \end{pmatrix}^t A ,
\end{equation}
where  $t$ denotes transposition and 
\[A = 
 \begin{pmatrix} 0  & a^0_1 & a^0_2 & \cdots & a^0_N & a^0_{N+1}\\
0  & 0 & a^1_2 & \cdots & a^1_N & a^1_{N+1}\\
0  & 0 & 0& \cdots & a^2_N & a^2_{N+1}\\
\vdots  & \vdots & \vdots & \ddots &\vdots\\
0  & 0 & 0 &\cdots &0  & a^{N}_{N+1}\\
0  & 0 & 0 &\cdots &0  & 0
 \end{pmatrix}.
 \]
The condition that the Euler forms and $H$ are cocycles yields
\begin{equation}\label{eq:sphere to massey}
\begin{aligned}
0 & =  d\left(\begin{pmatrix}  d\Psi \\ 0\end{pmatrix}^t+ \begin{pmatrix}\mathbf{0} \\H \end{pmatrix}^t\right)\\
& =d \left(\begin{pmatrix}\Psi \\0 \end{pmatrix}^t A \right)\\
& =\left(d \begin{pmatrix}\Psi \\ 0\end{pmatrix}\right) ^tA - \begin{pmatrix}\bar{\Psi} \\ 0\end{pmatrix}^t dA\\
& =\left( \begin{pmatrix}\Psi \\ 0\end{pmatrix}^t A - \begin{pmatrix}\mathbf{0} \\H \end{pmatrix}^t\right) A - \begin{pmatrix}\bar{\Psi} \\ 0\end{pmatrix}^tdA\\
& =\begin{pmatrix}\Psi \\ 0\end{pmatrix}^t A  A - \begin{pmatrix}\bar{\Psi} \\ 0\end{pmatrix}^tdA\\
\end{aligned}
\end{equation}
where we have used properties \eqref{eq:properties} at several places, we used \eqref{eq:the matrix A}  in the second and fourth equality and  in the last equality we used that  $\begin{pmatrix}\mathbf{0} \\H \end{pmatrix}^t A = 0$ because the last row of $A$ only has zeros.

Finally we observe that all the entries of the vector
\[\begin{pmatrix}\Psi \\ 0\end{pmatrix}^t A\]
are even forms, except maybe for the last, which depends on the degree of $H$.  And since the last entry is annihilated when we take the product 
\[\begin{pmatrix}\Psi \\ 0\end{pmatrix}^t A A,\]
we have that
\[\begin{pmatrix}\Psi \\ 0\end{pmatrix}^t AA =-\left(\overline{ \begin{pmatrix}\Psi \\ 0\end{pmatrix}^t A} \right)A=  \begin{pmatrix}\bar\Psi \\ 0\end{pmatrix}^t \bar{A} A.\]
Hence \eqref{eq:sphere to massey} becomes
\[ \begin{pmatrix}\bar\Psi \\ 0\end{pmatrix}^t (\bar{A} A-dA)= 0.\]

The equation above shows that $ dA -\overline{A} A =0$ , hence $A$  corresponds to   vanishing Massey product data  for the Massey product $\IP{[a^0_1],\dots, [a^N_{N+1}]}$, as in Remark \ref{rem:matrix massey}.
\end{proof}

From now on, we will often write \eqref{eq:the matrix A} in place of the equivalent and more notation heavy equations \eqref{eq:psi and H}.

\begin{remark}
In the Massey product produced above, $a_1^0$ is an even cocycle, since it is a representative for the Euler class of an odd sphere bundle. The cocycles $a_j^{j-1}$ are odd for $j\leq N$  as they are the push-forward of even classes along an odd-dimensional fibrations. The final cocycle $a_{N+1}^N$ can be either even or odd, depending on whether $H$ is an odd or even cocycle, respectively.
\end{remark}

It is important to notice that the iterated sphere bundle gives rise to vanishing Massey product data which contains much more information than the Massey product alone.

\begin{example}
Consider the following 2-iterated circle bundle over the 3-torus, $T$ (which has cohomology quasi-isomorphic to $(\wedge\{\theta_1,\theta_2,\theta_3\},d=0)$):
\begin{itemize}
\item Let $E_1 \to T$ be the trivial circle bundle and let $\psi_1$ with $d\psi_1 = 0$ be the global angular form for $E_1$;
\item Let $E_2 \to E_1$ be the circle bundle with Chern class $\theta_1 \theta_2$, that is, we let  $\psi_2$ with $d\psi_2 = \theta_1 \theta_2$ be the global angular form for $E_2$;
\item Let $H  = \psi_1 \theta_2 \theta_3$ be the cocycle in $E_2$.
\end{itemize}
Then, the following the recipe from Proposition \ref{prop:iterated => Massey}, we have
\[[a_1^0 ]= [d\psi_1]=0,\quad [a_2^1] = \pi_{1*}[d\psi_2]  =0, \quad \mbox{ and } [a_3^2] = \pi_{2*}[H]  =  0,\]
but the Massey product data obtained from this bundle is
\[ a_2^0 = \theta_1 \theta_2\qquad a_3^1 = \theta_2\theta_3.\]
So nontrivial bundles can still be associated to the trivial product, but via nontrivial data as illustrated in Example \ref{ex:many roads to zero}.
\end{example}

We obtain a converse to Proposition \ref{prop:iterated => Massey} in two steps. First we deal with the algebraic aspect which corresponds to the study of  iterated Hirsch extensions by odd elements. Aftewards we  get around the distinction between (integral) cohomology classes and  Euler classes of oriented sphere bundles. 

\begin{proposition}\label{prop:Massey => iterated 1}
Let $M$ be a manifold and  $\{a^i_j\}_{1\leq i < j \leq  N+1}$ be vanishing Massey product data. Assume further that
\begin{itemize}
\item  $a_1^0\in C^\bullet(M;\Z)$ is an even cocycle on $M$,
\item  for $i=1,\dots, N$,  $a_{i}^{i-1} \in C^\bullet(M;\Z)$ are odd cocycles and
\item  $a_{N+1}^N\in C^\bullet(M;\Z)$ is a cocycle.
\end{itemize}
Then  we let $C_0 =  C^\bullet(M;\Z)$  and for $j =1, \dots, N$ and form successive Hirsch extensions $C_j$  of $C_{j-1}$ by $V_j= \IP{\psi_j}$
where we declare
\[d \psi_j = a_j^0 +  \sum_{i=1}^{j-1} \psi_i a_j^i.
\]
Then 
\[H = a_{N+1}^0 +  \sum_{i=1}^{N} \psi_i a_{N+1}^i.\]
is a cocycle in $C_{N}$.
\end{proposition}
Notice that we are claiming that the proposed expressions for $d \psi_j$  and $H$ are  cocycles. The proof  is just a matter of reversing the steps of the previous proposition.

Notice that in the proposition above the cocycle $d\psi_1$ corresponds to an integral cohomology class on $M$, but it may not correspond to the Euler class of a sphere  bundle over $M$. The same observation applies to the cocycle $d\psi_j$  in the complex $C_{j-1}$, so the algebraic construction above may not correspond to an iterated sphere bundle.   Next, we tackle the question of  how to bridge this gap.

\begin{proposition}
In the same situation as in Proposition \ref{prop:Massey => iterated 1}, let $A = (a^i_j)$ be the matrix containing the vanishing Massey product data and for constants $c_j$, $j=0,1,\dots N-1$,  let 
\begin{equation}\label{eq:P}
P = \begin{pmatrix}
1 & 0 & 0 &\dots &0 & 0\\
0 & c_1 & 0 &\dots &0 & 0\\
0 & 0 & c_1 c_2  &\dots &0 &0\\
\vdots& & & \ddots& &\vdots\\
0 & 0 & 0 & \dots &  c_1 c_2 \dots c_{N-1} &0\\
0 & 0 & 0 & \dots &  0 & 1
\end{pmatrix}.
\end{equation}
Then  there is a collection of  nonzero natural numbers $c_j$, $j=0,1,\dots N-1$,  for which
\[P^{-1} A P\]
 is the data of a vanishing Massey product  arising from a level-one iterated sphere bundle as described in Proposition \ref{prop:iterated => Massey}. If  for all $j \leq N$ the class $[d\psi_j]\in C_{\Psi^<j}(M;\Z)$  is Euler class of sphere bundle over $E_{j-1}$, we can take $c_j=1$ .
\end{proposition}
\begin{proof}
Letting $\tilde{A} = P^{-1}A P$. The matrix $\tilde{A}$  is also a strictly  upper  triangular matrix whose entries are integral cochains and satisfies
\[d\tilde{A} - \bar{\tilde A} \tilde{A} = P^{-1}(dA - \bar{A} A) P = 0.\]
Hence $\tilde{A}$ is also Massey product vanishing data but now for $c_i  a_{i-1}^{i}$. Following the previous proposition, the successive Hirsch extensions corresponding to $\tilde{A}$ are related to the extensions corresponding to $A$ by
\[d\tilde{\psi}_j = (\Pi_{i=1}^jc_i) d \psi_j.\]
Since given an integral cohomology class, there is always a multiple of it which is the Euler class of a sphere bundle \cite{BV03},  we  can pick  $c_1$ so that class $d\tilde\psi_1 = c_1d\psi_1$ corresponds to the Euler class of a sphere bundle and  then  pick  $c_2$ so that $d\tilde\psi_2  = c_2 c_1d\psi_2$ corresponds to a class of a sphere bundle and so on.  Since the class $[H]$ does not have to arise from a sphere bundle, we do not need to scale it, so we can take the last entry of the matrix $P$  to be 1.
%
%
%
\end{proof}

For ease of reference, we add up the results of the last  three propositions under one roof:
\begin{theorem}\label{theo:iterated <=> left Massey}
Let $(E_j)_{j\leq N}$  be a level-one iterated sphere bundle with respect to a transgressive generating set $\Psi = \{\psi_{1}, \dots, \psi_N\}$, 
and let $H \in C_{\Psi}^l$ be a level-one $l$-cocycle. Letting $\psi_0 = 1$, we have
\[
d\begin{pmatrix}\Psi \\0 \end{pmatrix}^t +\begin{pmatrix}\mathbf{0} \\H \end{pmatrix}^t= \begin{pmatrix}\Psi \\0 \end{pmatrix}^t A .
\]
Then, the matrix $A$ satisfies 
\[dA - \bar{A} A =0,\]
that is, it corresponds to vanishing Massey product data. 

Conversely,  given a matrix $A = (a^i_j)$, with  $a^i_j \in C^\bullet(M;\Z)$,  corresponding to a vanishing Massey product with
\begin{itemize}
\item  $a_1^0$ an even cocycle on $M$,
\item  for $i=1,\dots, N$,  $a_{i}^{i-1}$  odd cocycles and
\item  $a_{N+1}^N$ a cocycle.
\end{itemize}
then there is a collection of  nonzero natural numbers $c_j$, $j=1,\dots N-1$,  for which
\[P^{-1}A P\]
 is the data of a vanishing Massey product  arising from a level-one iterated sphere bundle, where $P$ is given by \eqref{eq:P}.
 
The matrix $A$ itself corresponds to the Massey product data of an interated sphere bundle if the classes \eqref{eq:psi and H} are themselves representatives of Euler classes of sphere bundles.
\end{theorem}

The result is stronger in the case of torus bundles because
\begin{itemize}
\item every integral class is the Euler class of a circle bundle and
\item in this case the notion of transgressive level is independent of the transgressive set chosen. 
\end{itemize}

\begin{corollary}
There is a correspondence between level-one iterated circle bundles with a transgressive generating  set together with a level-one cocycle
 and vanishing Massey product data  $A = (a^i_j)$ 
with
\begin{itemize}
\item  $a_1^0 \in C^2(M;\Z)$,
\item  $a_{i}^{i-1} \in C^1(M;\Z)$ , for $i=1,\dots, N$  and
\item  $a_{N+1}^N$ a cocycle.
\end{itemize}
\end{corollary}

We finish this section with an aside of no consequence for the rest of the paper. 
\begin{remark} What happens with the nonzero Massey products?  These also arise in the context of iterated sphere bundles, but they answer a different question. Namely, a potentially nonzero product arises whenever we find a transgressive cochain for the fibration $E_N \to M$, that is  $H \in C_\Psi(M;\Z)$ and  $dH \in C(M;\Z)$
\end{remark}

\section{Order reversal duality}\label{sec:order reversal}

Assume we have Massey product data, $\{a_j^i\}_{0 \leq i < j \leq N+1}$  that shows that a Massey product vanishes
\[\IP{[a_1^0],\dots,[a_{N+1}^N]} = 0,\]
where $[a_1^0]$ and $[a_{N+1}^N]$  are even classes and the remaining classes are odd. Then Theorem \ref{theo:iterated <=> left Massey} allows us to re-interpret this data in terms of an iterated sphere bundle endowed with an odd-cocycle: the class $[a^1_0]$ gives information about the first sphere bundle, the subsequent classes $[a_{i+1}^i]$ give part of the information about iteration process and the class $[a_{N+1}^N]$ contains some information about the odd cocycle.

The idea of what comes next is simple: if we were to read the same data backwards, the class $[a_{N+1}^N]$ would correspond to information about the first sphere bundle, the classes $[a_{i+1}^i]$ would give part of the information about iteration process (now done backwards) and the class $[a_1^0]$ would have part of the information about an odd cocycle at the last step. This way of repackaging the same data would give rise to a (potentially) topologically distinct space. The point is that the same Massey product data, can be interpreted in two different ways as iterated sphere bundles.

Because we are being careful about the use of integral cochains, there is a small  obstacle in our way:  the space of singular cochains is not graded commutative so the data for $\IP{[a_1^0],\dots,[a_{N+1}^N]} $ does not translate into data for $\IP{[a_{N+1}^N],\dots,[a_{1}^0]}$. Instead, of actually reversing the order of the terms in the Massey product, we will see that the same product comes from the {\it right transgressive complex} of a different iterated sphere bundle. The results in this case are completely analogous to the ones obtained in the previous section, so next we just state the version  of Theorem \ref{theo:iterated <=> left Massey} with the relevant changes.

As before, assume that we have an iterated sphere bundle $(\hat{E}_k)_{k=1}^N$.  Assume that we have a  transgressive generating set of global angular cocycles $\hat\Psi = \{\hat{\psi}^{1},\dots,\hat{\psi}^{N}\}$ and use the right transgressive complexes of transgressive level 1 :
\[\hat{C}_{\hat{\Psi}^{\leq  j}} :=   C^\bullet(M;\Z) \tensor\wedge^{\leq 1} \IP{\hat{\psi}^{1},\dots, \hat{\psi}^{j}} ,\]
\[d\hat\psi^{j}\in \hat{C}_{\hat{\Psi}^{\leq  j-1}}.\] 

\begin{theorem}\label{theo:iterated <=> right Massey}
Let $(\hat{E}_j)_{j\leq N}$  be an iterated sphere bundle  which is level-one with respect to a right transgressive generating set $\hat{\Psi} = \{\hat{\psi}^{1}, \dots, \hat{\psi}^N\}$, 
and let $\hat{H} \in \hat{C}_{\hat{\Psi}^{\leq N}}^l$ be a level-one $l$-cocycle so we have
\begin{equation}\label{eq:dpsi hat}
\begin{aligned}
d\hat{\psi}^{j} &= \hat\varepsilon_{2n_j}= a^{N+1-j}_{N+1} + \sum_{i=1}^{j-1}  a_{N+1-i}^{N+1-j} \hat{\psi}^{i},\\
\hat{H}&= a_{N+1}^0 + \sum_{i=1}^{N}  a_{N+1-i}^0\hat{\psi}^{i}.
\end{aligned}
\end{equation}
Then, the collection $\{a^i_j\}_{0\leq i<j \leq N+1}$ is vanishing  Massey product data for the product
\[\IP{[a_1^0],\dots,[a_{N+1}^N]}.\]

Conversely, 
given vanishing Massey product data with
\begin{itemize}
\item  $a_1^0$ a cocycle on $M$,
\item  for $i=1,\dots, N$,  $a_{i}^{i-1}$ odd cocycles and
\item  $a_{N+1}^N$ an even cocycle.
\end{itemize}
there is a collection of  nonzero natural numbers $c_j$, $j=1,\dots N$,  for which
\[((\pi_{k=i}^{j-1}c_k) a_j^i)), \qquad \mbox{ for } 1\leq i \leq j \leq N+1\]
 is the data of a vanishing Massey product  arising from a level-one iterated sphere bundle.
 
The constants $c_j$ can be chosen to be 1 if the sequence of cocycles produced from the Massey product data using \eqref{eq:dpsi hat} are themselves representatives of Euler classes of sphere bundles.
\end{theorem}

\begin{proof}[Remark about the proof]
The proof is nearly identical to the proof of Theorem \ref{theo:iterated <=> left Massey}, so we only highlight the key step, namely how one obtains vanishing Massey product data from an iterated sphere bundle. Similarly to the previous case,  we  still set $\hat\psi^0 = 1$ and let 
 \[
 \begin{pmatrix} 
 0\\ \hat{\Psi}\end{pmatrix} = \begin{pmatrix} 0\\ \hat{\psi}^N\\ \vdots \\ \hat{\psi}^0\end{pmatrix} \quad \mbox{ and } \begin{pmatrix}  \hat{H}\\ \mathbf{0}\end{pmatrix} = \begin{pmatrix} \hat{H}\\ 0\\ \vdots \\ 0
 \end{pmatrix}.
 \]
 With these conventions,  equation \eqref{eq:dpsi hat} becomes
 \begin{equation}\label{eq:the matrix A II}
  d\begin{pmatrix}
0\\
\hat\Psi
 \end{pmatrix}
+\begin{pmatrix}
\hat{H}\\
\mathbf{0} 
\end{pmatrix} =
\hat{A}
\begin{pmatrix}0\\  \hat{\Psi}
\end{pmatrix}.
 \end{equation}
 where 
 \[\hat{A} = 
 \begin{pmatrix}
 0  & a^0_1 & a^0_2 & \cdots & a^0_N & a^0_{N+1}\\
0  & 0 & a^1_2 & \cdots & a^1_N & a^1_{N+1}\\
0  & 0 & 0 & \cdots & a^2_N & a^2_{N+1}\\
\vdots  & \vdots & \vdots & \ddots &\vdots\\
0  & 0 & 0 &\cdots &0  & a^{N}_{N+1}\\
0  & 0 & 0 &\cdots &0  & 0
 \end{pmatrix}.
 \]
 
  Then, using that the Euler forms and  $\hat{H}$ are cocycles, we have
 \begin{align*}
 0 &= d\left(   d\begin{pmatrix}
0\\
\hat\Psi
 \end{pmatrix}
+\begin{pmatrix}
\hat{H}\\
\mathbf{0} 
\end{pmatrix} \right)\\
 & = d\left(\hat{A}
\begin{pmatrix}0\\  \hat{\Psi}
\end{pmatrix}\right)\\
& = (d\hat{A})\begin{pmatrix}0\\  \hat{\Psi}
\end{pmatrix} - \bar{\hat{A}}\,  d\begin{pmatrix}0\\  \hat{\Psi}
\end{pmatrix}\\
& = (d\hat{A})\begin{pmatrix}0\\  \hat{\Psi}
\end{pmatrix} - \bar{\hat{A}} \left(\hat{A} \begin{pmatrix}0\\  \hat{\Psi}  \end{pmatrix} -  \begin{pmatrix}\hat{H}\\  \mathbf{0}  \end{pmatrix}\right)\\
&= (d\hat{A} - \bar{\hat{A}} \hat{A})\begin{pmatrix}0\\  \hat{\Psi}  \end{pmatrix} .
 \end{align*}

\end{proof}

\begin{definition}[Reverse dual]
A pair of level-one iterated sphere bundles  over a common base $M$ with transgressive generating sets $\Psi$ and $\hat{\Psi}$, endowed with odd cocycles, $(E,\Psi,H),(\hat{E},\hat{\Psi},\hat{H})\to M$ are \emph{reserve duals} if there is an invertible  diagonal matrix
\[P = \begin{pmatrix}
1 & & & &\\
& \lambda_1 &  &&\\
& & \ddots  &&\\
&  &  &\lambda_N&\\
&  &  & & 1\\
\end{pmatrix}, \quad \lambda_i \in \R
\] such that
\[\hat{A} = P^{-1} A P, \]
where $A$ and $\hat{A}$ are the Massey product matrices of $(E,\Psi,H),(\hat{E},\hat{\Psi},\hat{H})\to M$ with respect to their left and right transgressive complexes, respectively.   They are \emph{unimodular} reverse duals if  we can take $\lambda_i = \pm  1$ for all $i$.
\end{definition}

 
\begin{example}[T-dual circle bundles]
Working with de Rham cohomology, let $(E,H) \to M$ be a principal circle bundle endowed with a closed 3-form which represents an integral class $H$. Let $\psi$ be a connection on $E$ so that $d\psi = \varepsilon = a_1^0$ represents the Euler class of $E \to M$ and the transgressive complex of $E$ is $\Omega_\psi = \Omega(M)\tensor \wedge \IP{\psi}$. After changing $H$   by an exact form, we may assume it lies in the transgressive complex $\Omega_\psi$ and we write
\[H = h +  \psi \hat{\varepsilon} = :  a_2^0+ \psi a_2^1.\]
The Massey product data associated to this situation is therefore
\[a_1^0 = \varepsilon,\quad a_2^1 = \hat{\varepsilon},\quad a_2^0  = h \Rightarrow \IP{a_1^0,a_2^1} = d a_2^0 = dh.\]

The reverse dual to $(E,\psi,H)$ is a circle bundle with 3-form $(\hat{E},\hat{\psi},\hat{H})\to M$ with Chern class represented by  $a_2^1 = \hat\varepsilon$ and 3-form
\[\hat{H} = a_2^0 + a_1^0 \hat{\psi} = h + \varepsilon \hat{\psi}.\]
The triple $(\hat{E},\hat{\psi},\hat{H})$ is  the T-dual of $(E,\psi, H)$ constructed in \cite{BEM04}.
\end{example}

To make the sums and relations introduced above easier to digest it is useful to present  the data for an iterated sphere bundle and its reverse dual in a graphic way.  We focus on  unimodular case. If we are to make a table in which the columns correspond to the expansion of each $d\psi_i$ in terms of transgressive generators, then the lines would correspond to the reverse dual data as outlined in Table \ref{fig1}.

\begin{table}[h!!]
\begin{center}
\begin{tabular}{||c||c|c|c|c|c||l||}
\hline
\hline
& {\color{darkred}$d\psi_1\downarrow$} &{\color{darkred}$d\psi_2\downarrow$}& $\cdots$ &{\color{darkred}$ d\psi_{N}\downarrow$}& {\color{darkred}$ H\downarrow$}& \\
\hline
\hline
{\color{darkred}$\psi_0 $}&$a_1^0$	 &	$a_2^0$  & $\cdots$	& $a_N^0$	&  $a_{N+1}^0$ 	&{\color{darkgreen}$\leftarrow \hat{H}$}\\
{\color{darkred}$\psi_1 $}&	 & $a_2^1$	& $\cdots$	&$a_N^1$	& $a_{N+1}^1$	&{\color{darkgreen} $\leftarrow d\hat\psi^{N}$}\\ 
$\vdots$	    &	 		&		&$\ddots$	&	&	&$\vdots$\\ 
{\color{darkred}$\psi_{N-1}$}& &		& &$a^{N-1}_N$	&$a^{N+1}_{N-1}$	&{\color{darkgreen} $\leftarrow d \hat\psi^2$}\\ 
{\color{darkred}$\psi_{N}$}	& 	&		&  	&	&$a_{N+1}^N$&{\color{darkgreen} $\leftarrow d \hat\psi^1$}\\ 
\hline
\hline
& {\color{darkgreen}$\hat\psi^{N}$}& {\color{darkgreen}$\hat\psi^{N-1}$} & $\cdots$& {\color{darkgreen} $\hat\psi^{1}$} & {\color{darkgreen}$\hat\psi^0 $} &\\
\hline
\hline
\end{tabular}
\caption{The columns give the coefficients of  $\psi_0 := 1$, $\psi_1$,$\dots$, $\psi_n$  for the forms  $d\psi_1$, $d\psi_2$, $\dots$, $d\psi_n$  and $H$, while the lines  give the same information for $d\hat\psi^1$, $d\hat\psi^2$,  $\dots$, $d\hat\psi^n$  and $\hat{H}$. All entries below the diagonal have value 0 and were omitted.}\label{fig1}
\end{center}
\end{table}

\section{T-duality}\label{sec:T-duality}

Recall that for principal torus bundles, T-duality is defined as follows:

\begin{definition}
Let $(E,H)$, $(\hat{E},\hat{H})$ be principal  $k$-torus bundles over $M$ endowed with closed 3-forms $H \in \Omega^{3}_{cl}(E)$ and  $\hat{H} \in \Omega^{3}_{cl}(\hat{E})$. The bundles $(E,H)$ and $(\hat{E},\hat{H})$ are \emph{topological  T-duals} if there is an invariant form $F\in \Omega^{2}(E\times_M\hat{E})$ such that
\begin{itemize}
\item (Gerbe trivialization)
\[dF = p^*H - \hat{p}^*\hat{H}.\]
where $p\colon E\times_M\hat{E} \to E$ and $\hat{p}\colon E\times_M\hat{E}$ are the natural projections and
\item (nondegeneracy)  the map
\[\frak{t}\times \hat{\frak{t}} \to \R, \qquad (\gamma,\hat{\gamma}) \mapsto F(X_{\gamma},X_{\hat{\gamma}})\]
is nondegenerate, where $X_\gamma$ is the infinitesimal generator corresponding to the Lie algebra element $\gamma$ and $\frak{t}$ (resp. $\hat{\frak{t}}$) is the Lie algebra of the torus acting on $E$ (resp. $\hat{E}$).

The form $F$ is the \emph{T-duality kernel}. 
\end{itemize}
\end{definition}

We gather the information on T-dual spaces in the following commutative diagram
\begin{equation}\label{eq:basic setup}
\xymatrix@C=-18pt{
& (E\times_M \hat{E}, dF = p^*H - \hat{p}^*\hat{H})\ar[dl]^{p}\ar[dr]_{\hat{p}}&\\
(E,H)\ar[dr]^{\pi}& & (\hat{E},\hat{H})\ar[dl]_{\hat{\pi}}\\
& M&
}
\end{equation}

Changes in $H$ or $\hat{H}$ by exact forms can be incorporated into $F$ so being a T-dual pair only depends on the cohomology class of the 3-forms involved. Similarly typical consequences of spaces being T-dual such as isomorphisms of twisted cohomology, twisted K-theory and Courant algebroids are also preserved  by the action of 2-forms.

In the case above, since the T-duality kernel, $F$, has degree $2$, it also can at most have transgressive level 2 and the nondegeneracy condition in fact forces $F$ to be of transgressive level 2. If one allows the twisting forms, $H$ and $\hat{H}$, to be of higher degree, it is not a given that $F$ will be of transgressive level 2, even if we impose that the twists are of transgressive level 1. The general theory of T-duality can still be developed allowing for a T-duality kernel of higher transgressive level. But if we further impose that $F$ has transgressive level 2, then the nondegeneracy condition acquires the familiar form stated above (see \cite{Cav25},  Definition 3.2 , Corollary 3.5  and Example 4.9).

\begin{definition}
Let $(E,\Psi,H)\to M$ and $(\hat{E},\hat{\Psi},\hat{H})\to M$ be two level-one iterated sphere bundles with transgressive generating sets $\Psi$ and $\hat\Psi$ and level-one closed odd forms $H \in \Omega_\Psi(E)$ and $\hat{H}\in \Omega_{\hat{\Psi}}(\hat{E})$. The bundles  $(E,H)\to M$ and $(\hat{E},\hat{H})\to M$ form a {\it transgressive level-two  T-dual pair} if
\begin{itemize}
\item (Gerbe trivialization) there is a transgressive level-two form $F \in \Omega_{\Psi\cup\hat{\Psi}}(E\times_M \hat{E})$  such that
\[dF = p^*H - \hat{p}^*\hat{H}.\]
where $p\colon E\times_M\hat{E} \to E$ and $\hat{p}\colon E\times_M\hat{E}$ are the natural projections and
\item (nondegeneracy)  the map
\begin{equation}\label{eq:nondegeneracy F}
\IP{\Psi}^*\times \IP{\hat{\Psi}}^* \to \R, \qquad (v,\hat{v}) \mapsto F(v,\hat{v})
\end{equation}
is nondegenerate.
\end{itemize}
The spaces $(E,\Psi,H)\to M$ and $(\hat{E},\hat{\Psi},\hat{H})$ are \emph{unimodular} T-duals if the bilinear form \eqref{eq:nondegeneracy F} is unimodular.  
\end{definition}

It follows from \cite{Cav25}  that T-dual spaces have isomorphic twisted cohomology and Clifford--Courant algebroids, so here we focus on the question of existence, which is  our main result:

\begin{theorem}
Let $(E,\Psi,H)\to M$ be a first-level $N$-iterated sphere bundle with a left transgressive generating set, $\Psi$, and level-one closed odd form of a fixed degree in the transgressive complex of $E$ representing an integral class, $H$.  If  $(\hat{E},\hat{\Psi},\hat{H})$  is any (unimodular) reverse dual  bundle of $(E,\Psi,H)$ then $(\hat{E},\hat{\Psi},\hat{H})$ is a (unimodular) T-dual  of $(E,\Psi,H)$. 
\end{theorem}
\begin{proof}
Using the notation of Theorems \ref{theo:iterated <=> left Massey} and \ref{theo:iterated <=> right Massey}, we let  $A$  and $\hat{A}$ be the matrices of Massey product data associated to  $(E,\Psi,H)$ and $(\hat{E},\hat{\Psi},\hat{H})$. By hypothesis, there is a diagonal matrix
\[P = \begin{pmatrix}
1 & & & &\\
& \lambda_1 &  &&\\
& & \ddots  &&\\
&  &  &\lambda_N&\\
&  &  & & 1\\
\end{pmatrix}
\]
such that $\hat{A} =  P^{-1}A P$. Then
\begin{equation}\label{eq:H-hatH}
\begin{aligned}
H -\hat{H} & = \begin{pmatrix}\mathbf{0}\\ H\end{pmatrix}^t  P \begin{pmatrix}0 \\ \hat\Psi\end{pmatrix} - \begin{pmatrix}\bar{\Psi} \\ 0\end{pmatrix}^t  P \begin{pmatrix}\hat{H}\\ \mathbf{0}\end{pmatrix} \\
&=  
\left(\begin{pmatrix}\bar\Psi\\ 0\end{pmatrix}^tA - d\begin{pmatrix} \Psi\\ 0\end{pmatrix}^t \right) P \begin{pmatrix}0 \\ \hat\Psi\end{pmatrix} - \begin{pmatrix}\bar\Psi \\ 0\end{pmatrix}^t  P \left(P^{-1}A P\begin{pmatrix}0\\ \hat{\Psi}\end{pmatrix} -d\begin{pmatrix}0 \\ \hat\Psi\end{pmatrix}\right) \\
&= - d\begin{pmatrix} \Psi\\ 0\end{pmatrix}^t  P \begin{pmatrix}0 \\ \hat\Psi\end{pmatrix}  +\begin{pmatrix}\bar\Psi \\ 0\end{pmatrix}^t P d\begin{pmatrix}0 \\ \hat\Psi\end{pmatrix}\\ 
& = - d\left(\begin{pmatrix} \Psi\\ 0\end{pmatrix}^t  P \begin{pmatrix}0 \\ \hat\Psi\end{pmatrix} \right) \\
& = -\sum_{i=1}^N \lambda_i\psi^i \hat{\psi}_{N+1-i},
\end{aligned}
\end{equation}
where in the first equality we used that the first and last entries of $P$ are 1 and that $\psi_0 = \hat{\psi}^0 =1$ and in  the second equality we used \eqref{eq:the matrix A} and \eqref{eq:the matrix A II}.
\end{proof}


\section{Examples}\label{sec:examples}
\begin{example}\label{ex:Heisenberg}
The 3-dimensional Heisenberg manifold, $M$, is the quotient of the Heisenberg group
\[\mathbb{H} = \left\{\begin{pmatrix} 1 & x & z\\
0 & 1 & y\\
0 & 0 & 1
\end{pmatrix}\colon x,y,z \in \R\right\}
\]
by left action of the lattice  $\Gamma$ of matrices  with integral coefficients, or more explicitly,
\[(x,y,z) \sim (x+1,y, z+y)\sim (x,y+1,z)\sim(x,y,z+1).\]

The (co)tangent bundle of $M$ is trivializable with global frames given by
\[TM = \IP{\{\del_x,\del_y + x \del_z, \del_z\}},\quad T^*M = \IP{\{dx,dy,dz - x dy\}}\]

The Heisenberg manifold is an iterated 2-torus bundle over the circle with projection
\[\begin{pmatrix} 1 & x & z\\
0 & 1 & y\\
0 & 0 & 1
\end{pmatrix}  \in \mathbb{H}/\Gamma \mapsto  y \in \R/\Z.\]
The global angular forms are
\[\psi_1 = dx,\qquad \psi_2 = dz -xdy\]
Any nonzero 3-form on $M$ will have transgressive level 2, so  to comply with the conditions of the construction, we must take $H =0$. With these choices we have
\[d\psi_1 = 0 =:a_1^0,\quad d\psi_2 = 0 - \psi_1 dy =: a_2^0 + \psi_1 a_2^1, \quad H = 0 = : a_3^0 +  \psi_1 a_3^1 + \psi_2 a_3^2,\]
that is, the only nonzero entry in the matrix $A$ is $a^1_2 =- dy$.
For these choices the reverse dual has the same Massey product data as indicated in Table \ref{tab:table2}.
\begin{table}[h!!]
\begin{center}
\begin{tabular}{||c||c|c|c||l||}
\hline
\hline
&{\color{darkred}$d\psi_1 \downarrow$}  &{\color{darkred}$d\psi_2 \downarrow$}& {\color{darkred}$H \downarrow$}& \\
\hline
\hline
{\color{darkred}$\psi_0$}&0  & 0 & $0$ &{\color{darkgreen}$\leftarrow \hat H$}\\
{\color{darkred}$\psi_1$}&0  & $-dy$ & 0 &{\color{darkgreen} $\leftarrow d\hat\psi^2$}\\ 
{\color{darkred}$ \psi_2$}& 0  & 0 & 0& {\color{darkgreen} $\leftarrow d \hat \psi^1$}\\
\hline
\hline
&  {\color{darkgreen}$\hat\psi^2 $} & {\color{darkgreen} $\hat\psi^1 $}& {\color{darkgreen} $\hat\psi^0$}&\\
\hline
\hline
\end{tabular}
\caption[Table]{The rows give the coefficients of  $1$, $\psi^1$, $\psi^2$ for the forms  $d\psi^1$, $d\psi^2$ and $H$, while the lines give the same information for $d\hat\psi_1$, $d\hat\psi_2$ and $\hat{H}$.}\label{tab:table2}
\end{center}
\end{table}

Notice that there is a slight difference between the transgressive generators of $\mathbb{H}\to S^1$ and its T-dual. Indeed, for $\mathbb{H}$ we had $d\psi_2 = - \psi_1 dy$, while for the dual we must have $d\hat\psi^2 = - dy\hat\psi^1$.  One readily sees that the dual is also the Heisenberg manifold but with ``reverse monodromy'', that is, for the dual we use the following projection to the circle:
\[\begin{pmatrix} 1 & x & z\\
0 & 1 & y\\
0 & 0 & 1
\end{pmatrix}  \in \mathbb{H}/\Gamma \mapsto  -y \in \R/\Z.\]

This choice makes $\mathbb{H}/\Gamma$ self-T-dual. Notice that already in this example, a vanishing 3-form $H$ does not induce a topologically trivial T-dual bundle.
\end{example}

\begin{example}
Consider the $n$-dimensional nilmanifold  $M$ given by the data $(0,0,e^{12},e^{13}, \dots, e^{1n-1})$. This means that the cotangent bundle of $M$ is generated by a global frame $\{e^1,e^2,\dots, e^n\}$ with
\[de^1 = de^2 = 0, \quad de^i = e^{1}\wedge e^{i-1} =: e^{1(i-1)}, \mbox{ for } i = 3, \dots, n.\]
where we used the shorthand $e^{ij} = e^i \wedge e^j$.

Endow $M$ with 3-form $e^{12n}$.
The manifold $M$ is an iterated torus bundle over the 2-torus with $\psi^i = e^{i+2}$ for $i=1,2, n-2$. Hence, once again we can read of the Massey product data and include it in a table (Table 3). One readily sees that $M$ is self-T dual in a similar way as in the previous example (with monodromy reversal obtained by composing the projection map $M \to T^2$ with the antipodal map).
\begin{figure}[h!!]
\begin{center}
\begin{tabular}{|c|||c|c|c|c|c||l|}
\hline
\hline
& {\color{darkred}$d\psi_1\downarrow$} &{\color{darkred}$d\psi_2 \downarrow$}& $\cdots$  & {\color{darkred}$ d\psi_{n-2} \downarrow$}& {\color{darkred}$ H\downarrow$}& \\
\hline
\hline
{\color{darkred}$\psi_0$}&$e^{12}$	 &	  	&	&	&	&{\color{darkgreen}$\leftarrow \hat H$}\\
{\color{darkred}$\psi_1$}&		 	& $e^1$	& 	&	&	&{\color{darkgreen} $\leftarrow d\hat\psi^{n-2}$}\\ 
$\vdots$& 		 	& & $\ddots$	&	&	&$\vdots$\\ 
{\color{darkred}$\psi_{n-3}$}&	 		&		& &$e^1$&	&{\color{darkgreen} $\leftarrow d\hat\psi^2$}\\ 
{\color{darkred}$\psi_{n-2}$}&	 		&		& 	&	&	$e^{12}$ &{\color{darkgreen} $\leftarrow d\hat\psi^1$}\\ 
\hline
\hline
& {\color{darkgreen}$\hat\psi^{n-2}$}& {\color{darkgreen}$\hat\psi^{n-3}$} & $\dots$ & {\color{darkgreen} $\hat\psi^{1} $}& {\color{darkgreen}$\hat\psi^0$} &\\
\hline
\hline
\end{tabular}
\caption[Table]{The lines give the coefficients of  $1$, $\psi_1$, ..., $\psi_{n-2}$  for the forms  $d\psi_1$, ..., $d\psi_{n-2}$  and $H$, while the columns give the same information for $d\hat\psi^1$, ..., $d\hat\psi^{n-2}$ and $\hat{H}$. Entries with value 0 were omitted.}
\end{center}
\end{figure}
\end{example}

\begin{example}[Nodal curve] In this example we study T-duality for the nodal curve using two different models for a neighbourhood of the nodal curve.
 
Building on Example \ref{ex:Heisenberg}, consider $M = (0,1)\times \mathbb{H}/\Gamma $ with coordinates $(x,y,z)$ in $\mathbb{H}$ and corresponding equivalence class $[x,y,z]\in  \mathbb{H}/\Gamma$ and  $r\in \R_+$. Just as in the previous example, $M$ is an iterated torus bundle over the punctured disc with  projection
\[M\to (0,1) \times S^1, \quad \mu(r, [x,y,z]) = (r, [x]),\] 
and global angular forms $\psi_1 = dy$  and $\psi_2= dz - xdy$, so $d\psi_1 =0$ and $d\psi_2 = \psi_1 \wedge dx$. Just as in the previous  example, the space $M$ is self T-dual and we have $d\hat\psi^1 =0$ and $d\hat\psi^2 =  dx \wedge \hat\psi^1.$

The manifold $M$ of this example is a torus fibration over the punctured disc, $\mathbb{D}^* = (0,1) \times S^1$ and one readily sees that the monodromy around the circle on the base is a Dehn twist. So $M$ is a differentiable model for the neighbourhood of a nodal/focus-focus fiber with the singular fiber itself removed. 

Next, we endow $M$ with three different geometric structures to illustrate the effects of T-duality on those. First we consider  two complex structures, described by global trivializations of their canonical bundle. We let
\begin{equation}\label{ex:tate complex}
\Omega_0 = (d\log r  +  i dy)\wedge (\log r \psi_1 + i\psi_2).
\end{equation}
The form $\Omega_0$ is  decomposable and satisfies $\Omega_0 \wedge \overline{\Omega}_0  \neq 0$, hence it determines an almost complex structure. Since
\[d\Omega_0 = -(d\log r  +  i dy)\wedge (d\log r \psi_1 + idy\ \psi_1) =-(d\log r  +  i dy)^2 \psi_1 = 0,\]  
 the complex structure is integrable. For this complex structure, the fibers of  $M\to \mathbb{D}^*$ are complex  tori. We can transport this complex structure  via T-duality:
\begin{align*}
\tau (\Omega_0)& := \int_{\psi_1}\int_{\psi_2} e^{-(\psi_1\hat\psi^2 + \psi_2 \hat \psi^1)}\Omega_0\\
&= (d\log r +  i dy)\wedge (\log r \hat\psi^1 + i\hat\psi^2)
\end{align*}
showing that the complex structure is self-dual, once we take into account the monodromy reversal discussed in Example \ref{ex:Heisenberg}.

Another complex structure we can put on $M$ has canonical bundle trivialized by
\begin{equation}\label{ex:totally real}
\Omega_1 = ( dy + i \psi_1)\wedge (d\log r + i\psi_2).
\end{equation}
This form once again is readily seen to determine an integrable complex structure. This time, the fibers of  $M \to \mathbb{D}^*$ are totally real and as expected, its T-dual is a symplectic structure and $M \to \mathbb{D}^*$ is a Lagrangian fibration 
\[\tau (\Omega_1) =  e^{-i(dy \hat\psi^2 + d\log r \hat \psi^1)}.\]
Of course, reversing the T-duality we see a Lagrangian fibration generating a complex structure as its T-dual.

Finally, we consider the symplectic structure $\omega = d\log r  dy + \psi_1 \psi_2$, which makes the fibers of $M \to \mathbb{D}^*$ symplectic. Its T-dual is given by
\[\tau (e^{i\omega}) =  e^{-i(dy \hat\psi^2 + d\log r \hat \psi^1)} = -i e^{i(d\log r dy - \hat \psi^1\hat\psi^2)}.\]
which is again of symplectic type.

For completeness we  spell out the correspondence between $M$ and a common holomorphic model for a neighbourhood of a nodal curve, namely the Tate curve family. The Tate curve family, $\tilde{M}$,  is the quotient of $\mathbb{D}^*\times \C^*$ by the $\Z$-action generated by $ (q,z) \sim (q,qz)$. We can write an explicit diffeomorphism between $\tilde{M}$  and $M$:
\[[r_1e^{2 \pi i\theta_1}, r_2 e^{2 \pi i\theta_2}] \in \tilde{M} \mapsto \left(r_1,\left[\frac{\log r_2}{\log r_1},\theta_1, \theta_2\right]\right) \in  M.\]
The identifications of the form $e^{i\theta} =e^{i(\theta+2 \pi)}$ correspond to the identifications  $(x,y,z) \sim  (x,y+1,z)\sim (x,y,z+1) $. while the identification $(q,z) \sim (q,qz)$ corresponds to $(x,y,z) \sim  (x+1,y,z+y)$.

The complex structure on the Tate curve family determined by the complex coordinates $(q,z)$ corresponds (up to factors of $2\pi$) to the complex structure $\Omega_0$ above.

%
\end{example}

\section{T-duality beyond  iterated sphere bundles}\label{sec:beyond}

In this final section we disentangle T-duality from {\it iterated} sphere bundles  and Massey products, providing a framework that aims to be useful for more general  sphere bundles.  To do so, we introduce a generalisation of global angular form and transgressive generating set and  of level-one cochains for a given ``transgressive generating set''. We then state a T-duality theorem in this new setup.

\begin{definition}
Given a fiber bundle $S^{2n_1+1} \times \dots \times S^{2n_N-1}\cdots E \stackrel{\pi}{\to} M$ a \emph{semi-transgressive} generating set for $E$ is a collection $\Psi$ of cochains $\psi_j \in C^{2n_j-1}(E;\Q)$ such that
\begin{itemize}
\item $\pi_*(\psi_{2n_N-1}\cup \dots \cup \psi_{2n_1-1}) = 1$,
\item  $d\psi_j \in C^{\bullet}(E;\Q)\tensor \wedge \IP{\psi_1,\dots,\psi_N}$.
\end{itemize}
\end{definition}

To keep the notation concise, it is useful to once again denote $\psi_0 = 1$. Similarly to the previous discussion we can introduce the notion of level-one cochains, in which case we can use integral cochains instead of rational ones:

\begin{definition}
Given a fiber bundle $S^{2n_1+1} \times \dots \times S^{2n_N-1}\cdots E \stackrel{\pi}{\to} M$, a  semi-transgressive generating set $\Psi$  for $E$ has \emph{level one} if  $\psi_j  \in C^{2n_j-1}(E;\Z)$ and 
\[d\psi_j = \sum_{j=0}^N\psi_ia_j^i, \quad \mbox{ with } a_j^i \in C^{\bullet}(M;\Z).\]
\end{definition}

Next, we endow $E$ with a level-one cocycle $H = \sum_{j=0}^N \psi_j a_{N+1}^j$. As before,  we can pack the information about the semi transgressive generating set and the cocycle $H$  neatly in matrix notation. If we form the  vector $\Psi^t = (\psi_0 ~\dots~ \psi_N)$, then we have 
\begin{equation}\label{eq:matrix notation left}
(d\Psi^t~H) = (\Psi^t~0) A,
\end{equation}
where now $A$ is a $(N+2)\times (N+2)$ matrix and  it  is not required to be upper triangular because we do not assume that $E$ is an iterated sphere bundle. Still, since $d\psi_0=0$, the first column of $A$ vanishes, further since there is no $\psi_{N+1}$ in the generating set  but we are writing the equation above with square matrices, we set the bottom line of $A$ to zero as well. 

The existence of  a left semi-transgressive generating set for $E$ is a genuine restriction on the topology of the bundle $E$. Still, if it exists, $A$ satisfies a ``flatness'' condition.
\begin{lemma}
Let $(E\to M,H,\Psi)$ be a multiple sphere bundle with a semi-transgressive generating set and a level-one closed odd form $H$. Then the matrix $A$ defined by  \eqref{eq:matrix notation left} satisfies
\[dA - \bar{A}A =0.\]
\end{lemma}
The proof is identical to the proof of Proposition \ref{prop:iterated => Massey}.

The same argument holds for a set of  level-one, right, semi-transgressive generators. We keep the indexing from  convention from \eqref{eq:dpsi hat}, that is, we let
\begin{align*}
d\hat{\psi}^{j} &= \hat\varepsilon_{2n_j}=\sum_{i=0}^{N}  a_{N+1-i}^{N+1-j} \hat{\psi}^{i},\\
\hat{H}&= \sum_{i=0}^{N}  a_{N+1-i}^0\hat{\psi}^{i}.
\end{align*}

With this setup complete, we can state the relevant result on T-duality.

\begin{theorem}
Let $(E,H)\to M$ and $(\hat{E},\hat{H})\to M$ be multiple-sphere bundles, let $\Psi$ and $\hat{\Psi}$ be left and right  level-one semi-transgressive sets for $E$ and $\hat{E}$ respectively for which $H$ and $\hat{H}$ are level-one forms. Let $A$ and $\hat{A}$ be the matrices of forms defined by
 \begin{equation}\label{eq:d^2 =0}
 (d\Psi^t~H) = (\Psi^t~0) \,A,\quad \begin{pmatrix}\hat{H} \\d\hat{\Psi}\end{pmatrix} = \hat{A} \,\begin{pmatrix}0\\\hat{\Psi}\end{pmatrix},
 \end{equation}
 where $\Psi^t = (1 ~ \psi_1~\dots~\psi_N)$ and $\hat{\Psi}^t = ( \hat\psi^N~\dots~\hat\psi^1~1)$. If there is an invertible matrix of closed even forms $P$ such that
  \begin{equation}\label{eq:assumption?}
  P = \begin{pmatrix}
  1 & \mathbf{0}^t&0\\
\mathbf{0} & \tilde{P} &\mathbf{0}\\
0& \mathbf{0}^t & 1
  \end{pmatrix}
  \end{equation}
  and
  \[\hat{A} = P^{-1} A P,\]
 then $(E,H)$ and $(\hat{E},\hat{H})$ are a T-dual pair. 
\end{theorem}
\begin{proof}
The key computation of the proof is identical to \eqref{eq:H-hatH}, which we repeat below for completeness:
\begin{align*}
H -\hat{H} & = \begin{pmatrix}\mathbf{0}\\ H\end{pmatrix}^t  P \begin{pmatrix}0 \\ \hat\Psi\end{pmatrix} - \begin{pmatrix}\bar{\Psi} \\ 0\end{pmatrix}^t  P \begin{pmatrix}\hat{H}\\ \mathbf{0}\end{pmatrix} \\
&=  
\left(\begin{pmatrix}\bar\Psi\\ 0\end{pmatrix}^tA - d\begin{pmatrix} \Psi\\ 0\end{pmatrix}^t \right) P \begin{pmatrix}0 \\ \hat\Psi\end{pmatrix} - \begin{pmatrix}\bar\Psi \\ 0\end{pmatrix}^t  P \left(P^{-1}A P\begin{pmatrix}0\\ \hat{\Psi}\end{pmatrix} -d\begin{pmatrix}0 \\ \hat\Psi\end{pmatrix}\right) \\
&= - d\begin{pmatrix} \Psi\\ 0\end{pmatrix}^t  P \begin{pmatrix}0 \\ \hat\Psi\end{pmatrix}  +\begin{pmatrix}\bar\Psi \\ 0\end{pmatrix}^t P d\begin{pmatrix}0 \\ \hat\Psi\end{pmatrix}\\ 
& = - d\left(\begin{pmatrix} \Psi\\ 0\end{pmatrix}^t  P \begin{pmatrix}0 \\ \hat\Psi\end{pmatrix} \right) \\
& = -d\left(\begin{pmatrix} \psi_1 & \cdots & \psi_N\end{pmatrix} \tilde{ P} \begin{pmatrix}\hat\psi^N \\ \vdots \\ \hat\psi^1\end{pmatrix} \right).
\end{align*}
 The first and last equality hold because of the assumption on the specific form of the matrix $P$ and the second equality follows from \eqref{eq:d^2 =0}. Since $\tilde{P}$ is invertible, the form 
 \[F= \begin{pmatrix} \psi_1 & \cdots & \psi_N\end{pmatrix} \tilde{ P} \begin{pmatrix}\hat\psi^N \\ \vdots \\ \hat\psi^1\end{pmatrix} \]
 satisfies the nondegeneracy condition in the definition of T-duality.
%
\end{proof}

While the theorem gives us a way to  identify T-dual pairs beyond iterated sphere bundles, it does not address the  existence question.

\bibliographystyle{hyperamsplain-nodash}
\bibliography{references}

\end{document}